\documentclass[10pt]{article}
\title{$L$-Fuzzy Semi-Preopen Operator in $L$-Fuzzy Topological Spaces}
\author{ \normalsize A. Ghareeb\smallskip\\
{\it\normalsize Department of Mathematics, Faculty of Science, South
Valley University, Qena,  Egypt.}\\ {\normalsize E-mail: nasserfuzt@aim.com.}}
\usepackage{time}
\usepackage{amsthm}
\usepackage{color,xypic}
\usepackage{hyperref}
\theoremstyle{definition}
\newtheorem{defn}{Definition}[section]
\newtheorem{thm}{Theorem}[section]

\newtheorem{lem}{Lemma}[section]

\theoremstyle{remark}

\begin{document}
\date{}
\maketitle
\begin{abstract}In this paper, we give the concept of $L$-fuzzy Semi-Preopen operator in $L$-fuzzy topological spaces, and use them to score $L$-fuzzy SP-cmpactnness in $L$-fuzzy topological spaces. We also study the relationship between $L$-fuzzy SP-compactness and SP-compactness in $L$-topological spaces.
\medskip

\medskip

\end{abstract}

\section{\large Introduction}
C. L. Chang \cite{Chang} introduced and developed the concept of fuzzy topological spaces
based on the concept of a fuzzy set introduced by Zadeh in \cite{zadeh}. Since then, various important notions in the classical topology such as compactness and continuity have been extended to Chang's fuzzy topological spaces.\medskip

In 1980, H\"{o}hle \cite{hohle} introduced the concept of the fuzzy measurable spaces with the idea of giving degrees in [0,1] to some topological terms rather than $0$ and $1$. In 1991, from a logical point view, Ying \cite{ming} introduced the concept of fuzzifying topology and gave its base and subbase, which is established on the crisp sets not on the fuzzy set. \medskip

It is valued to note that Ying \cite{ming} also introduced another fuzzy topological space which is called bifuzzy topological space. In fact, this structure was not new in fuzzy topology, since it is already defined by Kubiak \cite{Kubiak} and \v{S}ostak \cite{sostak}.  In 1992, Ramadan \cite{ramadan} has been defined the same structure  and gave it the name smooth fuzzy topology. Finally, this structure is called  $L$-fuzzy topological space where $L$ is an appropriate lattice (see \cite{hohleandRodBOOK}). Briefly speaking, an $L$-fuzzy topology on a set $X$ assigns to every $L$-subset on $X$ a certain degree of being open, other than being definitely open or not.\medskip

In order to generalize the concepts of Semiopen and Preopen $L$-subsets, Shi \cite{4a} introduced the notions of Semiopen operator and Preopen operator in $L$-fuzzy topological spaces. He introduced the notions of $L$-fuzzy Semicontinuous, $L$-fuzzy irresolute, $L$-fuzzy Precontinuous, $L$-fuzzy Preirresolute functions in terms of Semiopen and Preopen operators and discussed some of its properties.\medskip

In this paper, we introduce the concept of $L$-fuzzy Semi-Preopen operator in $L$-fuzzy topological spaces and study some of its properties. Then we introduce and study the concept of $L$-fuzzy SP-compactness in $L$-fuzzy topological spaces. Several characterizations based on $L$-fuzzy Semi-Preopen operator are presented for $L$-fuzzy SP-compactness.

\section{\large Preliminaries}
Throughout this paper $(L,\le ,\bigwedge ,\bigvee ,')$  is a
complete DeMorgan algebra, $X$ is a nonempty set. $L^{X} $ is the
set of all $L$-subsets  on $X$. The smallest element and the largest element in $L^{X} $ are denoted by $\underline \bot$ and $\underline \top$, respectively.  A complete lattice $L$ is a complete Heyting algebra if it satisfies the following infinite distributive law: For all $a\in L$ and all $B\subset L$, \[a\wedge\bigvee B=\bigvee\{a\wedge b\,|\,b\in B\}.\]

An element $a$ in  $L$ is called a prime element if $a\ge b\wedge c$
implies $a\ge b$ or $a\ge c$. An element $a$ in $L$ is called
co-prime if $a'$ is prime \cite{6}. The set of non-unit prime
elements in  $L$ is denoted by $P(L)$. The set of non-zero co-prime
elements in $L$ is denoted by $M(L)$.\medskip

The binary relation $\ll $ in  $L$ is defined as follows: for $a$,
$b\in L$,  $a\ll b$ if and only if for every subset  $D\subseteq
L$, the relation $b\le sup \,D$ always implies the existence of
$d\in D$ with  $a\le d$ \cite{7}. In a completely distributive DeMorgan algebra $L$, each element  $b$ is a sup of $\{ a\in L|a\ll
b\} $. A set $\{ a\in L|a\ll b\}$ is called the greatest minimal
family of  $b$ in the sense of \cite{9,10}, denoted by $\beta (b)$,
and $\beta ^{*} (b)=\beta (b)\cap M(L)$.  Moreover, for $b\in L$, we
define $\alpha (b)=\{ a\in L|a'\ll b'\}$ and $\alpha ^{*}
(b)=\alpha (b)\cap P(L)$.\medskip

Let $f:X\rightarrow Y$ be a crisp mapping.
Then an L-fuzzy mapping $f_L^{\rightarrow}:L^{X}\rightarrow L^{Y}$ is
induced by $f$ as usual, i.e., $f_L^{\rightarrow}(A)(y)=\bigvee_{x\in X,\quad f(x)=y}A(x)$  and  $f_L^{\leftarrow}(B)(x)=B(f(x))$.\medskip

An $L$-topological space  is a pair $(X, \tau)$, where  $\tau$ is a subfamily of $L^{X} $ which contains $\underline \bot$;  $\underline \top$ and is closed for any suprema and finite infima. $\tau$ is called an $L$-topology on  $X$. Members of $\tau$ are called open $L$-subsets and their complements are called closed $L$-subsets.\medskip

\begin{defn}[\cite{Kubiak,sostak}]A function $\mathcal{T}\colon L^{X}\rightarrow L$ is called an  L-fuzzy topology on $X$ if it satisfies the following conditions:
\begin{enumerate}
\item[(O1)]$\mathcal{T}(\underline{\bot})=\mathcal{T}(\underline{\top})=\top$.
\item[(O2)]$\mathcal{T}(A \wedge B)\geq \mathcal{T}(A)\wedge\mathcal{T}(B)$ for each $A$, $B\in L^{X}$.
\item[(O3)]$\mathcal{T}(\bigvee_{i\in\Gamma}A_{i})\geq \bigwedge_{i\in \Gamma}\mathcal{T}(A_{i})$ for any $\{A_{i}\}_{i\in\Gamma}\subset L^{X}$.
\end{enumerate}
\end{defn}
The pair $(X,\mathcal{T})$ is called an L-fuzzy topological spaces. $\mathcal{T}(A)$ can be interpreted as the degree to which $A$ is an open $L$-subset and $\mathcal{T}(A')$ will be called the degree of closedness of $A$, where $A'$ is the $L$-complement of $A$.  A function $f:(X,\mathcal{T})\rightarrow (Y,\mathcal{U})$ is said to be  continuous with respect to $L$-fuzzy topologies $\mathcal{T}$ and $\mathcal{U}$ if $\mathcal{T}(f_L^\leftarrow(B))\geq \mathcal{U}(B)$ holds for all $B\in L^Y$.\medskip

For  $a\in L$ and  the function $\mathcal{T}:L^{X}\rightarrow L$, we use the following notation from \cite{4}. \[\mathcal{T}_{[a]} =\{ A\in L^X|\mathcal{T}(A)\ge a\}.\]

\begin{thm}[\cite{Zhang}] Let $\mathcal{T}:L^X\rightarrow L$ be a function. Then the following conditions are equivalent:\begin{enumerate}\item[(1)] $\mathcal{T}$ is an $L$-fuzy topology on $X$;
\item[(2)] $\mathcal{T}_{[a]}$ is an $L$-topology on $X$, for each $a\in M(L)$.
\end{enumerate}
\end{thm}

\begin{defn}[\cite{13}]Let $(X,\mathcal{T})$ be an $L$-topological space, $a\in L_{\bot}$ and $G\in L^X$. A family $\mathcal{U}\subseteq L^X$ is called a
$\beta_{a}$-cover of $G$ if for any $x\in X$, it follows that
$a\in\beta(G'(x)\vee\bigvee_{A\in\mathcal{U}}A(x))$. $\mathcal{U}$
is called a strong $\beta_{a}$-cover of $G$ if $a\in\beta(\bigwedge_{x\in
X}(G'(x)\vee\bigvee_{A\in\mathcal{U}}A(x)))$.
\end{defn}

\begin{defn}[\cite{13}] Let $(X,\mathcal{T})$ be an $L$-topological space, $a\in L_{\bot}$ and $G\in L^X$. A family $\mathcal{U}\subseteq L^X$ is called a
$Q_{a}$-cover of $G$ if for any $x\in X$, it follows that $G'(x)\vee\bigvee_{A\in\mathcal{U}}A(x)\geq a$.
\end{defn}

It is obvious that a strong $\beta_{a}$-cover of $G$ is a
$\beta_{a}$-cover of $G$, and a $\beta_{a}$-cover of $G$ is a
$Q_{a}$-cover of $G$.

\begin{defn}[\cite{13}] Let $(X,\mathcal{T})$ be an $L$-topological space, $a\in L_\top$ and $G\in L^X$. A family $\mathcal{A}\subseteq L^X$ is said to be:
\begin{enumerate}
\item[(1)] an $a$-shading of $G$ if for any $x\in X$, $(G'(x)\vee\bigvee_{A\in\mathcal{A}}A(x))\not\leq a$.
\item[(2)] a strong $a$-shading of $G$ if $\bigwedge_{x\in X}(G'(x)\vee\bigvee_{A\in\mathcal{A}}A(x))\not\leq a$.
\item[(3)] an $a$-remote family of $G$ if for any $x\in X$, $(G(x)\wedge\bigwedge_{B\in\mathcal{A}}B(x))\not\geq a$.
\item[(4)] a strong $a$-remote family of $G$ if $\bigvee_{x\in X}(G(x)\wedge\bigwedge_{B\in\mathcal{A}}B(x))\not\geq a$.
\end{enumerate}
\end{defn}

\begin{defn}[\cite{Thakur}] Let $(X ,\mathcal{T})$ be an L-topological space. An $L$-subset $A\in L^X$ is called Semi-Preopen if there is a Preopen subset $B$ such that $B\leq A \leq cl(B)$.
\end{defn}

\begin{defn}[\cite{BAI}] Let $(X, \mathcal{T})$ be an $L$-topological space and $G\in L^X$. Then $G$ is called fuzzy SP-compact  if for every family $\mathcal{U} \subset L^X$ of Semi-Preopen $L$-subsets, it follows that \[\bigwedge_{x\in X} \left(G'(x)\vee \bigvee_{A\in {\mathcal{U}}} A(x)\right)\le \bigvee_{\psi \in 2^{{\mathcal{U}}} } \bigwedge_{x\in X} \left(G'(x)\vee \bigvee_{A\in \psi } A(x)\right).\]
\end{defn}

\begin{defn}[\cite{4c}] Let $(X,\mathcal{T})$ be an $L$-fuzzy topological space. For $A\in L^X$, define the mapping  $\mathcal{T}_p:L^X\rightarrow L$ by \[\mathcal{T}_p(A)=\bigwedge_{x_\lambda\ll A}\bigvee_{x_\lambda\ll B}\left\{\mathcal{T}(B)\wedge\bigwedge_{y_\mu\ll B}\bigwedge_{y_\mu\not\leq D\geq A}(\mathcal{T}(D'))'\right\}.\]Then $\mathcal{T}_p$ is called $L$-fuzzy preopen operator induced by $\mathcal{T}$, where $\mathcal{T}_p(A)$ can be regarded as the degree to which $A$ is preopen and $\mathcal{T}_p^*(A)=\mathcal{T}_p(A')$ can be regarded as the degree to which $A$ is Preclosed.
\end{defn}

\section{\large $L$-Fuzzy Semi-Preopen Operator}
\begin{defn} Let $(X,\mathcal{T})$ be an $L$-fuzzy topological space. For any $G\in L^X$, define the mapping $\mathcal{T}_{sp}:L^X\rightarrow L$ by\[\mathcal{T}_{sp}(A)=\bigvee_{B\leq A}\left\{\mathcal{T}_p(B)\wedge\bigwedge_{x_\lambda\ll A}\bigwedge_{x_\lambda\not\leq D\geq B}(\mathcal{T}(D'))'\right\}.\] Then $\mathcal{T}_{sp}$ is called the $L$-fuzzy Semi-Preopen operator induced by $\mathcal{T}$ and $\mathcal{T}_p$, where $\mathcal{T}_{sp}(A)$ can be regarded as the degree to which $A$ is Semi-Preopen and $\mathcal{T}_{sp}^*(A)=\mathcal{T}_{sp}(A')$ can be regarded as the degree to which $A$ is Semi-Preclosed.
\end{defn}

\begin{thm} Let $(X,\mathcal{T})$ be an $L$-fuzzy topological space and $A\in L^X$. Then $A\in (\mathcal{T}_{sp})_{[a]}$ if and only if $A$ is Semi-Preopen in $\mathcal{T}_{[a]}$, where $a\in M(L)$ and $(\mathcal{T}_{sp})_{[a]}=\{A\in L^X~|~\mathcal{T}_{sp}(A)\geq a\}$.
\end{thm}

\begin{proof}\begin{eqnarray*}A\in(\mathcal{T}_{sp})_{[a]}&\Leftrightarrow &\mathcal{T}_{sp}(A)\geq a\\
&\Leftrightarrow &\bigvee_{B\leq A}\left\{\mathcal{T}_{p}(B)\geq a\wedge\bigwedge_{x_\lambda\ll A}\bigwedge_{x_\lambda\not\leq D\geq B}(\mathcal{T}(D'))'\right\}\geq a\\
&\Leftrightarrow &\exists B\leq A\quad \mbox{such that}\quad \mathcal{T}_p(B)\geq a\quad \mbox{and}\quad \bigwedge_{x_\lambda\ll A}\bigwedge_{x_\lambda\not\leq D\geq B}(\mathcal{T}(D'))'\geq a\\
&\Leftrightarrow &\exists B\leq A\quad \mbox{such that}\quad B\in(\mathcal{T}_{p})_{[a]}\quad \mbox{and}\quad \bigwedge_{x_\lambda\ll A}\bigwedge_{x_\lambda\not\leq D\geq B}(\mathcal{T}(D'))'\geq a\\
&\Leftrightarrow &\exists B\leq A\quad \mbox{such that}\quad B\in(\mathcal{T}_{p})_{[a]}\quad \mbox{and for each}\quad x_\lambda\ll A,\quad \bigvee_{x_\lambda\not\leq D\geq B}\mathcal{T}(D')\leq a'\\
&\Leftrightarrow & \exists B\leq A \quad \mbox{such that}\quad  B\in(\mathcal{T}_p)_{[a]}\quad \mbox{and for each}\quad x_\lambda\ll A,\quad \mbox{if}\quad \bigvee_{D\geq B}\mathcal{T}(D')\geq a,\quad \mbox{then}\quad  x_\lambda\leq D\\
&\Leftrightarrow & \exists B\leq A\quad \mbox{such that}\quad B\in(\mathcal{T}_p)_{[a]}\quad \mbox{and}\quad A\leq cl(B)\\
&\Leftrightarrow & A\quad \mbox{is Semi-Preopen.}
\end{eqnarray*}
\end{proof}

\begin{defn} Let $(X,\mathcal{T}_1)$ and $(Y,\mathcal{T}_2)$ be two $L$-fuzzy topological spaces. The function $f:(X,\mathcal{T}_1)\rightarrow (Y,\mathcal{T}_2)$ is called:\begin{enumerate}
\item[(1)] $L$-fuzzy Semi-Precontinuous function if $\mathcal{T}_2(V)\leq(\mathcal{T}_1)_{sp}(f_L^\leftarrow(V))$ holds for each $V\in L^Y$.
\item[(2)] $L$-fuzzy Semi-Preirresolute if $(\mathcal{T}_2)_{sp}(V)\leq(\mathcal{T}_1)_{sp}(f_L^\leftarrow(V))$ holds for each $V\in L^Y$.
\end{enumerate}
\end{defn}

\begin{thm}Let $(X,\mathcal{T})$ and $(Y,\mathcal{U})$ be two $L$-fuzzy topological spaces. Then $f: (X,\mathcal{T})\rightarrow (Y,\mathcal{U})$ is $L$-fuzzy semi-precontinuous if and only if $f:(X,(\mathcal{T})_{[a]})\rightarrow (Y,(\mathcal{U})_{[a]})$ is $L$-semi-precontinuous for each $a\in M(L)$.
\end{thm}

\begin{proof}$(\Rightarrow)$ Let $V\in\mathcal{U}_{[a]}$, then $\mathcal{U}(V)\geq a$. Since $f:(X,\mathcal{T})\rightarrow (Y,\mathcal{U})$ is $L$-fuzzy semi-precontinuous, $\mathcal{T}_{sp}(f_L^\leftarrow(V))\geq\mathcal{U}(V)\geq a$, so $\mathcal{T}_{sp}(f_L^\leftarrow(V))\geq a$, therefore $f_L^\leftarrow(V)\in(\mathcal{T}_{sp})_{[a]}$, this implies that $f_L^\leftarrow(V)$ is Semi-Preopen in $(X,\mathcal{T}_{[a]})$. So that $f:(X,\mathcal{T}_{[a]})\rightarrow (Y,\mathcal{U}_{[a]})$ is $L$-semi-precontinuous.\medskip

$(\Leftarrow)$ Let $\mathcal{U}(V)\geq a$ for each $a\in M(L)$, then $V\in \mathcal{U}_{[a]}$. By the semi-precontinuity of $f:(X,\mathcal{T}_{[a]})\rightarrow (Y,\mathcal{U}_{[a]})$, we have $f_L^\leftarrow(V)\in(\mathcal{T}_{sp})_{[a]}$. Accordingly, $\mathcal{T}_{sp}(f_L^\leftarrow(V))\geq a$ for each $a\in M(L)\cap M(\mathcal{U}(V))$, where $M(\mathcal{U}(V))=\{a\in M(L)|a\leq\mathcal{U}(V)\}$. It follows that \[\mathcal{T}_{sp}(f_L^\leftarrow(V))\geq\bigvee M(\mathcal{U}(V))=\mathcal{U}(V).\]
\end{proof}

\begin{thm}Let $(X,\mathcal{T})$ and $(Y,\mathcal{U})$ be two $L$-fuzzy topological spaces. The function $f: (X,\mathcal{T})\rightarrow (Y,\mathcal{U})$ is $L$-fuzzy Semi-Preirresolute if and only if $f:(X,\mathcal{T}_{[a]})\rightarrow (Y,\mathcal{U}_{[a]})$ is $L$-Semi-Preirresolute for each $a\in M(L)$.
\end{thm}

\begin{proof}$(\Rightarrow)$ Suppose that $V$ is an $L$-Semi-Preopen in $(Y,\mathcal{U}_{[a]})$, then $V\in (\mathcal{U}_{sp})_{[a]}$, so $\mathcal{U}_{sp}(V)\geq a$. Since $f:(X,\mathcal{T})\rightarrow (Y,\mathcal{U})$ is $L$-fuzzy Semi-Preirresolute, $\mathcal{T}_{sp}(f_L^\leftarrow(V))\geq\mathcal{U}_{sp}(V)\geq a$, so $\mathcal{T}_{sp}(f_L^\leftarrow(V))\geq a$, therefore $f_L^\leftarrow(V)\in(\mathcal{T}_{sp})_{[a]}$, this implies that $f_L^\leftarrow(V)$ is Semi-Preopen in $(X,\mathcal{T}_{[a]})$. So that $f:(X,\mathcal{T}_{[a]})\rightarrow (Y,\mathcal{U}_{[a]})$ is $L$-Semi-Preirresolute.\medskip

$(\Leftarrow)$ Let $\mathcal{U}_{sp}(V)\geq a$ for each $a\in M(L)$, then $V\in(\mathcal{U}_{sp})_{[a]}$. Since $f:(X,\mathcal{T}_{[a]})\rightarrow (Y,\mathcal{U}_{[a]})$ is $L$-Semi-Preirresolute, $f_L^\leftarrow(V)\in(\mathcal{T}_{sp})_{[a]}$. Accordingly, $\mathcal{T}_{sp}(f_L^\leftarrow(V))\geq a$ for any $a\in M(L)\cap M(\mathcal{U}_{sp}(V))$, where $M(\mathcal{U}_{sp}(V))=\{a\in M(L)|a\leq\mathcal{U}_{sp}(V)\}$. It follows that \[\mathcal{T}_{sp}(f_L^\leftarrow(V))\geq\bigvee M(\mathcal{U}_{sp}(V))=\mathcal{U}_{sp}(V).\]
\end{proof}

\section{\large SP-Compactness in $L$-fuzzy Topological Spaces}

\begin{defn}Let $(X,\mathcal{T})$ be an $L$-fuzzy topological space. An $L$-subset $G\in L^X$ is called $L$-fuzzy SP-compact if for every family $\mathcal{P}\subseteq L^X$, it follows that \[\bigwedge_{F\in\mathcal{P}}\mathcal{T}_{sp}(F)\wedge\bigwedge_{x\in X}\left(G'(x)\vee\bigvee_{F\in\mathcal{P}}F(x)\right)\leq\bigvee_{\mathcal{Q}\in 2^{(\mathcal{P})}}\bigwedge_{x\in X}\left(G'(x)\vee\bigvee_{F\in\mathcal{Q}}F(x)\right).\]
Where $2^{(\mathcal{P})}$ denotes the set of all finite subfamily of $\mathcal{P}$.
\end{defn}

\begin{thm}Let $(X,\mathcal{T})$ be an $L$-fuzzy topological space. An $L$-subset $G\in L^X$ is called $L$-fuzzy SP-compact if for every family $\mathcal{W}\subseteq L^X$, it follows that
\[\bigvee_{F\in\mathcal{W}}(\mathcal{T}^*_{sp}(F))'\vee\bigvee_{x\in X}\left(G(x)\wedge\bigwedge_{F\in\mathcal{W}}F(x)\right)\geq\bigwedge_{\mathcal{H}\in 2^{(\mathcal{W})}}\bigvee_{x\in X}\left(G(x)\wedge\bigwedge_{F\in\mathcal{H}}F(x)\right).\]
\end{thm}

\begin{proof}Straightforward.
\end{proof}

\begin{thm}Let $(X,\mathcal{T})$ be an $L$-fuzzy topological space and $G\in L^X$. The following conditions are equivalent:\begin{enumerate}
\item[(1)] $G$ is an $L$-fuzzy SP-compact;
\item[(2)] For any $a\in M(L)$, each strong $a$-remote family $\mathcal{P}$ of $G$ with $\bigwedge_{F\in\mathcal{P}}\mathcal{T}_{sp}^*(F)\not\leq a'$ has a finite subfamily $\mathcal{H}$ which is a (strong) $a$-remote family of $G$;
\item[(3)] For any $a\in M(L)$, each strong $a$-remote family $\mathcal{P}$ of $G$ with $\bigwedge_{F\in\mathcal{P}}\mathcal{T}_{sp}^*(F)\not\leq a'$, there exists a finite subfamily $\mathcal{H}$ of $\mathcal{P}$ and $b\in\beta^*(a)$ such that $\mathcal{H}$ is a (strong) $b$-remote family of $G$;
\item[(4)] For any $a\in P(L)$, each strong $a$-shading $\mathcal{U}$ of $G$ with $\bigwedge_{F\in\mathcal{P}}\mathcal{T}_{sp}(F)\not\leq a$ has a finite subfamily $\mathcal{V}$ which is a (strong) $a$-shading of $G$;
\item[(5)] For any $a\in P(L)$, each strong $a$-shading $\mathcal{U}$ of $G$ with $\bigwedge_{F\in\mathcal{P}}\mathcal{T}_{sp}(F)\not\leq a$, there exists a finite subfamily $\mathcal{V}$ of $\mathcal{U}$ and $b\in\alpha^*(a)$ such that $\mathcal{V}$ is a (strong) $b$-shading of $G$;
\item[(6)] For any $a\in M(L)$ and $b\in\beta^*(a)$, each $Q_a$-cover $\mathcal{U}$ of $G$ with $\mathcal{T}_{sp}(F)\geq a$ ( for each $F\in\mathcal{U}$) has a finite subfamily $\mathcal{V}$ which is a $Q_b$-cover of $G$;
\item[(7)] For any $a\in M(L)$ and any $b\in\beta^*(a)$, $Q_a$-cover $\mathcal{U}$ of $G$ with $\mathcal{T}_{sp}(F)\geq a$ ( for each $F\in\mathcal{U}$) has a finit subfamily $\mathcal{V}$ which is a (strong) $\beta_a$-cover of $G$.
\end{enumerate}
\end{thm}

\begin{proof}Straightforward.
\end{proof}

\begin{thm}Let $(X,\mathcal{T})$ be an $L$-fuzzy topological space and $G\in L^X$. If $\beta(c\wedge d)=\beta(c)\wedge\beta(d)$ for each $c$, $d\in L$, then the following conditions are equivalent:\begin{enumerate}
\item[(1)] $G$ is $L$-fuzzy SP-compact;
\item[(2)] For any $a\in M(L)$, each strong $\beta_a$-cover $\mathcal{U}$ of $G$ with $a\in\beta\left(\bigwedge_{F\in\mathcal{U}}\mathcal{T}_{sp}(F)\right)$ has a finite subfamily $\mathcal{V}$ which is a (strong) $\beta_a$-cover of $G$;
\item[(3)] For any $a\in M(L)$, each strong $\beta_a$-cover $\mathcal{U}$ of  $G$ with $a\in\beta\left(\bigwedge_{F\in\mathcal{U}}\mathcal{T}_{sp}(F)\right)$, there exists a finite subfamily $\mathcal{V}$ of $\mathcal{U}$ and $b\in M(L)$ with $a\in\beta^*(b)$ such that $\mathcal{V}$ is a (strongly) $\beta_b$-cover of $G$.
\end{enumerate}
\end{thm}

\begin{proof}Straightforward.
\end{proof}

\section{\large Properties of $L$-fuzzy SP-Compactness}

\begin{defn} Let $(X,\mathcal{T})$ be an $L$-topological space, $a\in M(L)$ and $G\in L^X$. $G$ is said to be $a$-fuzzy SP-compact if and only if for each $b\in\beta(a)$, $Q_a$-Semi-Preopen cover $\mathcal{U}$ of $G$ has a finite subfamily $\mathcal{V}$ which is a $Q_b$-Semi-Preopen cover of $G$.
\end{defn}

\begin{thm}Let $(X,\mathcal{T})$ be an $L$-topological space. $G\in L^X$ is fuzzy SP-compact if and only if $G$ is $a$-fuzzy SP-compact for each $a\in M(L)$.
\end{thm}

\begin{proof}$(\Rightarrow)$ Suppose that $G$ is fuzzy SP-compact and for any $a\in L_{\top}$, $b\in\beta(a)$ and $\mathcal{U}$ is any $Q_a$-Semi-Preopen cover of $G$. Then, we have\[ \bigwedge_{x\in X}\left(G(x)'\vee\bigvee_{F\in\mathcal{U}}F(x)\right)\leq\bigvee_{\mathcal{V}\in 2^{(\mathcal{U})}}\bigwedge_{x\in X}\left(G(x)'\vee\bigvee_{F\in\mathcal{V}}F(x)\right)\] and $a\leq\bigwedge_{x\in X}\left(G(x)'\vee\bigvee_{F\in\mathcal{U}}F(x)\right)$, so that \[a\leq \bigvee_{\mathcal{V}\in 2^{(\mathcal{U})}}\bigwedge_{x\in X}\left(G(x)'\vee\bigvee_{F\in\mathcal{V}}F(x)\right).\] By $b\in\beta(a)$, we have \[b\leq\bigvee_{\mathcal{V}\in 2^{(\mathcal{U})}}\bigwedge_{x\in X}\left(G(x)'\vee\bigvee_{F\in\mathcal{V}}F(x)\right).\] Hence there exists $\mathcal{V}\in 2^{(\mathcal{U})}$ such that $b\leq\bigwedge_{x\in X}\left(G(x)'\vee\bigvee_{F\in\mathcal{V}}F(x)\right)$. This shows that $b$ is $Q_b$-Semi-Preopen cover of $G$.\medskip

$(\Leftarrow)$ Suppose that each $Q_a$-Semi-Preopen cover $\mathcal{U}$ of $G$ has a finite subfamily $\mathcal{V}$ which is a $Q_a$-Semi-Preopen cover of $G$ for each $b\in\beta(a)$. Then $a\leq\bigwedge_{x\in X}\left(G(x)'\vee\bigvee_{F\in\mathcal{U}}F(x)\right)$ implies that $b\leq\bigwedge_{x\in X}\left(G(x)'\vee\bigvee_{F\in\mathcal{U}}F(x)\right)$. Therefore $a\leq\bigwedge_{x\in X}\left(G(x)'\vee\bigvee_{F\in\mathcal{U}}F(x)\right)$ implies that $b\leq\bigvee_{\mathcal{V}\in 2^{(\mathcal{U})}}\bigwedge_{x\in X}\left(G(x)'\vee\bigvee_{F\in\mathcal{U}}F(x)\right)$. So $a\leq\bigwedge_{x\in X}\left(G(x)'\vee\bigvee_{F\in\mathcal{U}}F(x)\right)$ implies that $\bigvee_{b\in\beta(a)}b\leq\bigvee_{\mathcal{V}\in 2^{(\mathcal{U})}}\bigwedge_{x\in X}\left(G(x)'\vee\bigvee_{F\in\mathcal{U}}F(x)\right)$, i.e, $a\leq\bigwedge_{x\in X}\left(G(x)'\vee\bigvee_{F\in\mathcal{U}}F(x)\right)$ implies that \[a\leq\bigvee_{\mathcal{V}\in 2^{(\mathcal{U})}}\bigwedge_{x\in X}\left(G(x)'\vee\bigvee_{F\in\mathcal{U}}F(x)\right).\] Hence \[\bigwedge_{x\in X}\left(G(x)'\vee\bigvee_{F\in\mathcal{U}}F(x)\right)\leq\bigvee_{\mathcal{V}\in 2^{(\mathcal{U})}}\bigwedge_{x\in X}\left(G(x)'\vee\bigvee_{F\in\mathcal{V}}F(x)\right).\]
\end{proof}

\begin{thm}\label{thmthm}Let $(X,\mathcal{T})$ be an $L$-fuzzy topological space and $G\in L^X$. $G$ is an $L$-fuzzy SP-compact in $(X,\mathcal{T})$ if and only if $G$ is $a$-fuzzy SP-compact in $(X,\mathcal{T}_{[a]})$ for each $a\in M(L)$.
\end{thm}

\begin{proof}$(\Rightarrow)$ Since $G$ is $L$-fuzzy SP-compact in $(X,\mathcal{T})$, then for every family $\mathcal{U}\subseteq L^X$, we have \[\bigwedge_{F\in\mathcal{U}}\mathcal{T}_{sp}(F)\wedge\bigwedge_{x\in X}\left(G(x)'\vee\bigvee_{F\in\mathcal{U}}F(x)\right)\leq\bigvee_{\mathcal{V}\in 2^{(\mathcal{U})}}\bigwedge_{x\in X}\left(G(x)'\vee\bigvee_{F\in\mathcal{V}}F(x)\right).\] Hence for each $a\in M(L)$ and $\mathcal{U}\in(\mathcal{T}_{sp})_{[a]}$, we have that\[a\leq\bigwedge_{x\in X}\left(G(x)'\vee\bigvee_{F\in\mathcal{U}}F(x)\right)\Rightarrow a\leq\bigvee_{\mathcal{V}\in 2^{(\mathcal{U})}}\bigwedge_{x\in X}\left(G(x)'\vee\bigvee_{F\in\mathcal{V}}F(x)\right).\]
Thus for each $b\in\beta(a)$, there exists $\mathcal{V}\in 2^{(\mathcal{U})}$ such that $b\leq\bigwedge_{x\in X}\left(G(x)\vee\bigvee_{F\in\mathcal{V}}F(x)\right)$. i.e, for each $a\in M(L)$ and $b\in\beta(a)$, each $Q_a$-Semi-Preopen cover $\mathcal{U}$ of $G$ in $(X,\mathcal{T}_{[a]})$ has a finite subfamily $\mathcal{V}$ which is a $Q_a$-cover. Therefore for each $a\in M(L)$, $G$ is $a$-fuzzy SP-compact in $(X,\mathcal{T}_{[a]})$.\medskip

$(\Leftarrow)$ Suppose that for each $a\in M(L)$, $G$ is $a$-fuzzy SP-compact in $(X,\mathcal{T}_{[a]})$. Let $\mathcal{U}\subseteq L^X$ and $a\leq\bigwedge_{F\in\mathcal{U}}\mathcal{T}_{sp}(F)\wedge\bigwedge_{x\in X}\left(G(x)'\vee\bigvee_{F\in\mathcal{U}}F(x)\right)$. Then $a\leq\bigwedge_{F\in\mathcal{U}}\mathcal{T}_{sp}(F)$ and $a\leq \bigwedge_{x\in X}\left(G(x)'\vee\bigvee_{F\in\mathcal{U}}F(x)\right)$, i.e, $\mathcal{U}\subseteq (\mathcal{T}_{sp})_{[a]}$ and $a\leq \bigwedge_{x\in X}\left(G(x)'\vee\bigvee_{F\in\mathcal{U}}F(x)\right)$. Thus for each $b\in\beta(a)$, there exists $\mathcal{V}\in 2^{(\mathcal{U})}$ such that $b\leq\bigwedge_{x\in X}\left(G(x)'\vee\bigvee_{F\in\mathcal{V}}F(x)\right)$. So that $a\leq\bigvee_{\mathcal{V}\in 2^{(\mathcal{U})}}\bigwedge_{x\in X}\left(G(x)'\vee\bigvee_{F\in\mathcal{V}}F(x)\right)$. Therefore $G$ is $L$-fuzzy SP-compact in $(X,\mathcal{T})$.
\end{proof}

Analogous to Theorem 3.8 in \cite{BAI}, we can obtain the following lemma:

\begin{lem}\label{lemlemlem} Let $(X,\mathcal{T})$ be an $L$-topological space, $a\in M(L)$ and $G$, $H\in L^X$. If $G$ is $a$-fuzzy SP-compact and $H$ is Semi-Preclosed, then $G\wedge H$ is $a$-fuzzy SP-compact.
\end{lem}

\begin{thm}Let $(X,\mathcal{T})$ be an $L$-fuzzy topological space and $G\in L^X$. If $G$ is an $L$-fuzzy SP-compact and $\mathcal{T}_{sp}^*(H)=\top$, then $G\wedge H$ is $L$-fuzzy SP-compact.
\end{thm}
\begin{proof}This is immediate from Lemma \ref{lemlemlem}.
\end{proof}

Analogous to Theorem 3.7 in \cite{BAI}, we can obtain the following lemma:

\begin{lem}
Let $(X,\mathcal{T})$ be an $L$-topological space, $a\in M(L)$ and $G$, $H\in L^X$. If $G$, $H$ are $a$-fuzzy SP-compact, then $G\vee H$ is $a$-fuzzy SP-compact.
\end{lem}

\begin{thm}
Let $(X,\mathcal{T})$ be an $L$-fuzzy topological space and $G$, $H\in L^X$. If $G$, $H$ are $L$-fuzzy SP-compact, then $G\vee H$ is $L$-fuzzy SP-compact.
\end{thm}

\begin{proof}Straightforward.
\end{proof}

\begin{lem}\label{lemlem}Let $(X,\mathcal{T})$ and $(Y,\mathcal{U})$ be two $L$-topological spaces, $a\in M(L)$, $G\in L^X$ and $f:(X,\mathcal{T})\rightarrow (Y,\mathcal{U})$ be an $L$-Semi-Preirresolute function. If $G$ is $a$-fuzzy SP-compact in $(X,\mathcal{T})$, then $f_L^\rightarrow(G)$ is $a$-fuzzy SP-compact in $(Y,\mathcal{U})$.
\end{lem}

\begin{thm}Let $(X,\mathcal{T})$ and $(Y,\mathcal{U})$ be two $L$-fuzzy topological spaces,  $G\in L^X$ and $f:(X,\mathcal{T})\rightarrow (Y,\mathcal{U})$ be an $L$-fuzzy Semi-Preirresolute function. If $G$ is $L$-fuzzy SP-compact in $(X,\mathcal{T})$, then $f_L^\rightarrow(G)$ is $L$-fuzzy SP-compact in $(Y,\mathcal{U})$.
\end{thm}

\begin{proof} Since $G$ is $L$-fuzzy SP-compact in $(X,\mathcal{T})$, by Theorem \ref{thmthm}, $G$ is $a$-fuzzy SP-compact in $(X,\mathcal{T}_{[a]})$ for each $a\in M(L)$. By Theorem \ref{thmthm}, $f:(X,\mathcal{T}_{[a]})\rightarrow (Y,\mathcal{U}_{[a]})$ is $L$-Semi-Preirresolute. So that $f_L^\rightarrow(G)$ is $a$-fuzzy SP-compact in $(Y,\mathcal{U}_{[a]})$ (by Lemma \ref{lemlem}). Therefore $f_L^\rightarrow(G)$ is $L$-fuzzy SP-compact in $(Y,\mathcal{U})$.
\end{proof}


\begin{thebibliography}{10}

\bibitem{BAI}
S.-Z. Bai, \emph{A new notion of {SP}-compact ${L}$-fuzzy sets}, Proyecciones
  \textbf{25} (2006), 249--259.

\bibitem{Chang}
C.~L. Chang, \emph{Fuzzy topological spaces}, J. Math. Anal. Appl. \textbf{24}
  (1968), 39--90.

\bibitem{7}
P.~Dwinger, \emph{Characterizations of the complete homomorphic images of a
  completely distributive complete lattice i}, Indagationes
  Mathematicae(Proceedings) \textbf{85} (1982), 403--414.

\bibitem{6}
G.~Gierz and et~al., \emph{A compendium of continuous lattices}, Springer
  Verlag, Berlin, 1980.

\bibitem{hohleandRodBOOK}
U.~H{\"{o}}hle and S.~E. Rodabaugh, \emph{Mathematics of fuzzy sets: {L}ogic,
  {T}opology, and {M}easure theory}, vol.~3, Kluwer Academic Publishers,
  Boston/Dordrecht/London, 1999.

\bibitem{hohle}
U.~Hohle and A.~P. \v{S}ostak, \emph{Upper semicontinuous fuzzy sets and
  applications}, J. Math. Anal. Appl. \textbf{78} (1980), 659--670.

\bibitem{Zhang}
F.-G.~Shi J.~Zhang and C.-Y. Zheng, \emph{On ${L}$-fuzzy topological spaces},
  Fuzzy Sets and Systems \textbf{149} (2005), 473--484.

\bibitem{Kubiak}
T.~Kubiak, \emph{On fuzzy topologies}, Ph.{D} thesis, A. Mickiewicz, poznan,
  1985.

\bibitem{9}
Y.~M. Liu and M.~K. Luo, \emph{Fuzzy topology}, World Scientific, Singapore,
  1997.

\bibitem{ming}
Y.~Ming-Sheng., \emph{A new approach to fuzzy topology {(I)}}, Fuzzy sets and
  systems \textbf{39} (1991), 303--321.

\bibitem{ramadan}
A.~A. Ramadan, \emph{Smooth pretopogenous structures}, Fuzzy sets and systems
  \textbf{48} (1992), 371--375.

\bibitem{4c}
F.-G. Shi, \emph{Semiopenness and preopenness in ${L}$-fuzzy topological
  spaces}, (Submitted).

\bibitem{4}
\bysame, \emph{Theory of ${L}_{\beta}$-nested sets and ${L}_{\alpha}$-nested
  and their applications}, Fuzzy Systems and Mathematics (In Chinese)
  \textbf{4} (1995), 65--72.

\bibitem{4a}
\bysame, \emph{Semicompactness in ${L}$-topological spaces}, Int. J. Math.
  Math. Sci. \textbf{12} (2005), 1869–78.

\bibitem{13}
\bysame, \emph{A new form of fuzzy $\alpha$-compactness}, Mathematica Bohemica
  \textbf{131} (2006), 15--28.

\bibitem{sostak}
A.~P. {\v{S}}ostak, \emph{On a fuzzy topological structure}, Suppl. Rend. Circ.
  Matem. Palermo-Sir II \textbf{11} (1985), 89--103.

\bibitem{Thakur}
S.S.Thakur and S.Singh, \emph{On fuzzy semi-preopen sets and fuzzy semi
  precontinuity}, Fuzzy Sets and Systems \textbf{98} (1996), 383--391.

\bibitem{10}
G.-J. Wang, \emph{Theory of ${L}$-fuzzy topological space}, Shaanxi Normal
  University Press, Xi'an, 1988 (in Chinese).

\bibitem{zadeh}
L.~A. Zadeh, \emph{Fuzzy sets}, Inform. Control \textbf{8} (1965), 338--353.

\end{thebibliography}

\providecommand{\bysame}{\leavevmode\hbox to3em{\hrulefill}\thinspace}
\providecommand{\MR}{\relax\ifhmode\unskip\space\fi MR }
\providecommand{\MRhref}[2]{%
  \href{http://www.ams.org/mathscinet-getitem?mr=#1}{#2}
}
\providecommand{\href}[2]{#2}

\end{document}